\documentclass[11pt, leqno]{amsart}

\usepackage{
  amsmath,
  amsthm,
  amssymb,
  tikz,
  fancyhdr,
  enumerate,
  setspace,
  multirow,
}

\usetikzlibrary{arrows,calc,matrix,shapes}

\usepackage[center,small,sc]{caption}

\usepackage{multicol}

\theoremstyle{plain}
\newtheorem{theorem}{Theorem}[section]
\newtheorem{lemma}[theorem]{Lemma}

\newtheorem{corollary}[theorem]{Corollary}

\theoremstyle{definition}
\newtheorem{definition}[theorem]{Definition}
\newtheorem{example}[theorem]{Example}

\newtheorem{remark}[theorem]{Remark}
\numberwithin{equation}{section} \numberwithin{figure}{section}
\numberwithin{table}{section}

\newcommand{\matr}[2]{\left( \begin{matrix} #1 \\ #2 \end{matrix} \right)}

\newcommand{\seteq}{\mathbin{:=}}
\newcommand{\F}{\mathfrak{F}}

\newcommand{\wF}{\widetilde{\F}}

\newenvironment{red}{\relax\color{red}}{\relax}
\newenvironment{blue}{\relax\color{blue}}{\hspace*{.5ex}\relax}
\newenvironment{magem}{\relax\color{magenta}}{\relax}

\newcommand{\ber}{\begin{red}}
\newcommand{\er}{\end{red}}
\newcommand{\beb}{\begin{blue}}
\newcommand{\eb}{\end{blue}}
\newcommand{\bema}{\begin{magem}}
\newcommand{\ema}{\end{magem}}

\newcommand{\Z}{\mathbb{Z}}
\newcommand{\wB}{\widetilde{B}}

\newcommand{\rnum}[1]{#1}

\newcommand{\bu}[1]{ {\color{blue}#1}}

\tikzset{tab/.style={matrix of math nodes,column sep=-.4, row
sep=-.4,text height=8pt,text width=8pt,align=center}}

\title[Catalan triangle numbers and binomial coefficients]
{Catalan triangle numbers and \\ binomial coefficients}

\author[K.-H. Lee, S.-j. Oh]{Kyu-Hwan Lee$^{\star}$ and Se-jin Oh$^\dagger$}

\thanks{$^{\star}$This work was partially supported by a grant from the Simons Foundation (\#318706).}
\thanks{$^\dagger$This work was supported by NRF Grant \# 2016R1C1B2013135.}

\address{Department of Mathematics, University of Connecticut, Storrs, CT 06269, U.S.A.}
\email{khlee@math.uconn.edu}

\address{Department of Mathematics, Ewha Womans University, Seoul, 120-750, South Korea}
\email{sejin092@gmail.com}

\date{\today}

\begin{document}

\begin{abstract} The binomial coefficients and Catalan triangle numbers appear as weight multiplicities of the finite-dimensional simple Lie algebras and affine Kac--Moody algebras.
We prove that any binomial coefficient can be
written as weighted sums along rows of the Catalan triangle.
The coefficients in the sums form a triangular array, which we call the {\em alternating Jacobsthal triangle}.
We study various subsequences of the entries of the alternating Jacobsthal triangle and show that they arise in a variety of combinatorial constructions.
The generating functions of these sequences enable us to define their $k$-analogue of $q$-deformation. We show that this deformation also
gives rise to interesting combinatorial sequences. The starting point of this work is certain identities in the study of Khovanov--Lauda--Rouquier
algebras and fully commutative elements of a Coxeter group.
\end{abstract}

\maketitle

\section{Introduction}

It is widely accepted that
Catalan numbers are the most frequently occurring combinatorial numbers after the
binomial coefficients. As binomial coefficients can be defined inductively from the Pascal's triangle, so  can Catalan numbers from a triangular array of numbers whose entry in the $n^{\textrm{th}}$ row and $k^\textrm{th}$ column is denoted by $C(n,k)$ for $0 \leq k \leq n$. Set the first entry $C(0,0)=1$, and then each subsequent entry is the sum of the entry above it and the entry to the left.  All entries outside of the range $0\leq k \leq n$ are considered to be $0$. Then we obtain the array shown in \eqref{CT} known as {\em Catalan triangle} introduced by L.W. Shapiro \cite{Sh} in 1976.
Notice that Catalan numbers $C_n$ appear on the hypotenuse of the triangle, i.e. $C_n=C(n,n)$ for $n\ge 0$.

\begin{equation} \label{CT} \begin{array}{ccccccccc}
1\\
1 & 1\\
1 & 2 & 2\\
1 & 3 & 5 & 5\\
1 & 4 & 9 & 14 & 14\\
1 & 5 & 14 & 28 & 42 & 42\\
1 & 6 & 20 & 48 & 90 & 132 & 132\\
1 & 7 & 27 & 75 & 165 & 297 & 429 & 429\\
\vdots & \vdots &\vdots &\vdots &\vdots &\vdots &\vdots & \vdots& \ddots\\

\end{array}
\end{equation}

The first goal of this paper is to write each binomial coefficient
as weighted sums along rows of the Catalan triangle. In the first case, we take the sums along the $n^{\mathrm{th}}$ and $n+1^{\mathrm{st}}$ rows of the Catalan triangle,
respectively, and obtain $2$-power weighted sums to express a binomial coefficient. More precisely, we  prove:
\begin{theorem} \label{thm: smaller}
For integers $n\ge 1$ and $0 \le k \le  n+1$, we have the identities
\begin{equation} \label{main}
\begin{aligned}
\matr{n+k+1}{k} & =\sum_{s=0}^{k} C(n,s)2^{k-s}= \sum_{s=0}^{k} C(n+1,s)2^{\max(k-1-s,0)}
\end{aligned}
\end{equation}
and
\begin{equation} \label{mm} \matr{n+1}{\lceil (n+1)/2 \rceil}= \sum_{k=0}^{\lceil n/2 \rceil } C(\lceil n/2 \rceil,k)
2^{\max( \lfloor n/2 \rfloor -k,0)}. \end{equation}
\end{theorem}

It is quite intriguing that the two most important families of combinatorial numbers are related in this way. By replacing $2$-powers with $x$-powers in the identities, we  define
Catalan triangle polynomials and make a conjecture on stacked directed animals studied in \cite{BR} (see Section \ref{subsec:Cpoly}).
 It is also
interesting that $\matr{n+1}{\lceil (n+1)/2 \rceil}$ appearing in \eqref{mm} is exactly the
number of fully commutative, involutive elements of the Coxeter group of type $A_n$. (See (4.2) in \cite{Stembridge1998}.)
%
Actually, a clue to the identities \eqref{main} and \eqref{mm} was found in the study of the homogeneous representations of
Khovanov--Lauda--Rouquier algebras and the fully commutative elements of type $D_n$ in the paper \cite{FL14}
of the first-named author and G. Feinberg, where they proved the following:

\begin{theorem} \cite{FL14} \label{thm: D_n packet decom-1}
For $n \ge 1$, we have
\begin{align} \label{eq: D_n packet decom-1}
\dfrac{n+3}{2}C_n = \sum_{k=0}^{n-1}C(n,k)2^{|n-2-k|} , 
\end{align}
where $C_n$ is the $n^{\mathrm{th}}$ Catalan number.
\end{theorem}
We note that $\frac{n+3}{2}C_n -1$ is the number of the fully commutative elements of type $D_n$. (See \cite{Stembridge1998}).  The identity \eqref{eq: D_n packet decom-1}
is obtained by decomposing the set of fully commutative elements of type $D_n$ into {\em packets}.
Likewise, we expect interesting combinatorial interpretations and representation-theoretic applications of the identities \eqref{main} and \eqref{mm}.
In particular, $C(n,k)$ appear as weight multiplicities of finite-dimensional simple Lie algebras and affine Kac--Moody algebras of types $A$ and $C$ \cite{KLO,Ts,TW}.

To generalize Theorem \ref{thm: smaller}, we use other rows of the Catalan triangle and there appear a natural sequence of numbers $A(m,t)$, defined by
\begin{align*}
A(m,0)=1, \quad   A(m,t)=A(m-1,t-1) - A(m-1,t),
\end{align*}
to yield the following result:
\begin{theorem} \label{thm: along n+a-1}
For any $n > k \ge m \ge t \ge 1$, we have
\begin{align}\label{eq-nmk-1}
\matr{n+k+1}{k} & = \sum_{s=0}^{k-m}C(n+m,s)2^{k-m-s} + \sum_{t=1}^{m}A(m,t)C(n+m,k-m+t).
\end{align}
In particular, when $k=m$,
we have
\begin{equation}\label{eq:combAC-1}
\matr{n+k+1}{k}  = \sum_{t=0}^{k}A(k,t)C(n+k,t).
\end{equation}
\end{theorem}

The sequence consisting of $A(m,t)$ is listed as A220074 in the On-line Encyclopedia of Integer Sequences (OEIS).
However, the identity \eqref{eq:combAC-1} does not seem to have been known.

The identity \eqref{eq:combAC-1} clearly suggests that the triangle consisting of the numbers $A(m,t)$ be considered as a transition triangle
from the Catalan triangle to the Pascal triangle. We call it the {\em alternating Jacobsthal triangle}. The triangle has  (sums of) subsequences of the
entries with interesting combinatorial interpretations.  In particular, diagonal sums are related to the Fibonacci numbers and horizontal sums are
related to the Jacobsthal numbers.

\medskip

The second goal of this paper is to study a $k$-analogue of $q$-deformation of the Fibonacci and Jacobsthal numbers through a $k$-analogue of the
alternating Jacobsthal triangle.  This deformation is obtained by putting the parameters $q$ and $k$ into the generating functions of these numbers.
Our constructions give rise to different polynomials than the Fibonacci and Jacobsthal polynomials which can be found in the literature (e.g. \cite{HBJ,Koshy}).

For example, the $k$-analogue $J_{k,m}(q)$ of $q$-deformation of the Jacobsthal numbers is given by the generating function

\[ \dfrac{x(1-qx)}{(1-kq^2x^2)(1-(q+1)x)} = \sum_{m=1}^{\infty} J_{k,m}(q)x^{m}.
\]
When $q=1$ and $k=1$, we recover the usual generating function $\dfrac x {(1+x)(1-2x)}$ of the Jacobsthal numbers.

Interestingly enough, sequences given by special values of this deformation have various combinatorial interpretations. For example,
the sequence  \[(J_{2,m}(1) )_{m\ge 1} = (1, 1, 4, 6, 16, 28, 64, 120, \dots )\] is listed as $A007179$ in OEIS and has the interpretation
as the numbers of equal dual pairs of some integrals studied in \cite{Heading}. (See Table 1 on p.365 in \cite{Heading}.)
 Similarly, many
subsequences of a $k$-analogue of the alternating Jacobsthal triangle are found to have combinatorial meanings.
See the triangle \eqref{T3}, for example.
%

\medskip

An outline of this paper is as follows. In the next section, we  prove Theorem \ref{thm: smaller} to obtain Catalan triangle expansions of binomial coefficients
as $2$-power weighted sums. We also introduce Catalan triangle polynomials and study some of their special values.
In Section \ref{sec:AJ}, we prove Theorem \ref{thm: along n+a-1} and investigate the alternating Jacobsthal triangle to obtain generating functions and  meaningful
subsequences. The following section is concerned about $q$-deformation of the Fibonacci and Jacobsthal numbers. The last section is devoted to the study of a $k$-analogue
of the $q$-deformation of the Fibonacci and Jacobsthal numbers using the $k$-analogue of the alternating Jacobsthal triangle.

\subsection*{Acknowledgments} We would like to thank Jang Soo Kim for helpful comments. The second-named author would like to thank the faculty and staff of the Department of Mathematics at the University of Oregon, where
part of this paper was completed. In particular, he is thankful to Benjamin Young for helpful discussions on topics related to this paper.

\medskip

\section{Catalan expansion of binomial coefficients} \label{sec:CE}

In this section, we prove expressions of binomial coefficients as $2$-power weighted sums along rows of the Catalan triangle. Catalan trapezoids are introduced
for the  proofs. In the last subsection, Catalan triangle polynomials are defined and some of their special values will be considered.

\subsection{Catalan triangle}

We begin with a formal definition of the Catalan triangle numbers.

\begin{definition} \label{def: Catalan}
For $n\geq0$ and $0\leq k \leq n$, we define the $(n,k)$-{\em Catalan triangle number} $C(n,k)$ recursively by
  \begin{equation}\label{sum rule}
    C(n,k) =
    \left\{
    \begin{array}{cl}
        1 & \textrm{ if } n=0 ;\\
        C(n,k-1)+C(n-1,k) & \textrm{ if } 0<k<n; \\
        C(n-1, 0) & \textrm{ if } k=0; \\
        C(n, n-1) & \textrm{ if } k=n,\\
    \end{array}
    \right.
  \end{equation}
and define the $n^{\textrm{th}}$ {\em Catalan number} $C_n$ by \begin{equation*} \label{ccc}
C_n=C(n,n) \qquad \text{ for } n \ge 0.
\end{equation*}
\end{definition}

The closed form formula for the Catalan triangle numbers is well known: for $n \geq 0$ and $0 \leq k \leq n$,
\begin{equation*} \label{cat form}
  C(n,k) = \frac{(n+k)!(n-k+1)}{k!(n+1)!}.
\end{equation*}
In particular, we have \[ C_n = \frac 1 {n+1} \binom {2n} n .\]

\begin{theorem} \cite{FL14} \label{thm: D_n packet decom}
For $n \ge 1$, we have
\begin{align} \label{eq: D_n packet decom}
\dfrac{n+3}{2}C_n = \sum_{k=0}^{n-1}C(n,k)2^{|n-2-k|} . 
\end{align}
\end{theorem}
As mentioned in the introduction,  $\frac{n+3}{2}C_n -1$ is the number of the fully commutative elements of type $D_n$ (\cite{Stembridge1998}) and the identity \eqref{eq: D_n packet decom} is obtained by decomposing the set of fully commutative elements of type $D_n$ into {\em packets}.

\begin{theorem} \label{thm: main1}
For $n \in \mathbb{Z}_{\ge 0}$, we have
\begin{align} \label{eq: ceil form}
\matr{n+1}{\lceil (n+1)/2 \rceil}= \sum_{s=0}^{\lceil
n/2 \rceil } C(\lceil n/2 \rceil,s) 2^{\max( \lfloor n/2 \rfloor
-s,0)}.
\end{align}
\end{theorem}

\begin{proof}
Set $\mathcal{Q}_{n} \seteq \matr{n+1}{\lceil (n+1)/2 \rceil}$ for convenience.
Assume $n=2k$ for some $k \in \mathbb{Z}_{\ge 0}$. Then we have
$$\mathcal{Q}_{2k} = \matr{2k+1}{k+1}.$$
By \eqref{eq: D_n packet decom} in Theorem \ref{thm: D_n packet decom}, we have
\begin{align} \label{eq: s format}
\sum_{s=0}^{k-1}C(k,s)2^{|k-2-s|}=
\dfrac{k+3}{2}C_k=\dfrac{k+3}{2k+2}\matr{2k}{k}.
\end{align}
On the other hand,
$$ \sum_{s=0}^{k-1}C(k,s)2^{|k-2-s|}=\sum_{s=0}^{k-2}C(k,s)2^{k-2-s}+2C(k,k-1).$$
Note that $C(k,k-1)=C_k$. Multiplying  \eqref{eq: s format} by $4$,
we have
$$\sum_{s=0}^{k}C(k,s)2^{k-s}+5C(k,k-1) = \dfrac{2k+6}{k+1}\matr{2k}{k}.$$
Hence
\begin{align*}
\sum_{s=0}^{k}C(k,s)2^{k-s} & = \dfrac{2k+6}{k+1}\matr{2k}{k}-5C_k =
\dfrac{2k+1}{k+1}\matr{2k}{k}=\matr{2k+1}{k+1} = \mathcal Q_{2k}.
\end{align*}

Assume $n=2k+1$ for some $k \in \mathbb{Z}_{\ge 0}$. Then we have
$$\mathcal{Q}_{2k+1} = \matr{2k+2}{k+1}.$$
By replacing the $k$ in \eqref{eq: s format} with $k+1$, we have
\begin{align} \label{eq: s format 2}
\sum_{s=0}^{k}C(k+1,s)2^{|k-1-s|}=
\dfrac{k+4}{2}C_{k+1}=\dfrac{k+4}{2k+4}\matr{2k+2}{k+1}.
\end{align}
On the other hand,
$$ \sum_{s=0}^{k}C(k+1,s)2^{|k-1-s|}=\sum_{s=0}^{k-1}C(k+1,s)2^{k-1-s}+2C(k+1,k).$$
Note that $C(k+1,k)=C_{k+1}$. Multiplying \eqref{eq: s format 2} by
$2$, we have
\begin{align*}
\sum_{s=0}^{k-1}C(k+1,s)2^{k-s}+4C(k+1,k)&=\sum_{s=0}^{k+1}C(k+1,s)2^{\max(k-s,0)}+2C(k+1,k) \allowdisplaybreaks\\
&= \dfrac{k+4}{k+2}\matr{2k+2}{k+1}.
\end{align*}
Since
$$\dfrac{k+4}{k+2}\matr{2k+2}{k+1}-2C(k+1,k)=\matr{2k+2}{k+1} = \mathcal{Q}_{2k+1},$$
our assertion is true in this case as well.
\end{proof}

\begin{corollary} \label{cor: dual}
For $n \in \mathbb{Z}_{\ge 0}$, the dual form of \eqref{eq: ceil form} holds; i.e.,
$$\matr{n+1}{\lfloor (n+1)/2 \rfloor}= \sum_{s=0}^{\lfloor n/2 \rfloor } C(\lfloor n/2 \rfloor,s)
2^{\lceil n/2 \rceil -s}.$$ In particular, we have the following identity by replacing $n$ with $2n-1$:
$$\matr{2n}{n} = \sum_{s=0}^{n-1}C(n-1,s)2^{n-s}=\sum_{s=0}^{n} \dfrac{(n+s-1)!(n-s)}{s!n!}2^{n-s}=\sum_{s=0}^{n}\matr{n}{s}^2.$$
\end{corollary}

\begin{proof}
The assertion for $n=2k$ follows from the fact that
$\matr{2k+1}{k}=\matr{2k+1}{k+1}$ and for $n=2k+1$ follows from the
fact that $2 \times \matr{2k+1}{k}=\matr{2k+2}{k+1}$.
\end{proof}

From Theorem \ref{thm: main1}, we have, for $n=2k-1$ $(k \in
\mathbb{Z}_{\ge 1})$,
\begin{align*}
\matr{2k}{k}&= \sum_{s=0}^{k} C(k,s)2^{\max(k-1-s,0)} \allowdisplaybreaks\\
 & = \sum_{s=0}^{k-1} C(k,s)2^{k-1-s}+ \dfrac{1}{k+1}\matr{2k}{k}.
\end{align*}
Since $\matr{2k}{k}-\dfrac{1}{k+1}\matr{2k}{k}= \matr{2k}{k-1}$, we
obtain a new identity:
\begin{align} \label{eq: s-1 induction}
\matr{2k}{k-1}=\sum_{s=0}^{k-1} C(k,s)2^{k-1-s}.
\end{align}

More generally, we have the following identity which is an interesting expression of a binomial
coefficient $\matr{n+k+1}{k}$ as a $2$-power weighted sum of the Catalan triangle along the $n^{\textrm{th}}$ row.

\begin{theorem} \label{thm: larger}
For $0 \le k \le  n+1$, we have
\begin{equation} \label{eq-nk} \matr{n+k+1}{k}=\sum_{s=0}^{\min(n,k)} C(n,s)2^{k-s}. \end{equation}
\end{theorem}

\begin{proof}
We will use an induction on $|n+1-k|$. The cases when $k=n+1,n,n-1$ are already proved by Theorem \ref{thm: main1}, Corollary \ref{cor: dual} and \eqref{eq: s-1 induction}.
Assume that we have the  identity  \eqref{eq-nk} for $k+1<n$:
\begin{align} \label{eq: smaller}
\matr{n+k+2}{k+1}=\sum_{s=0}^{k+1} C(n,s)2^{k+1-s}.
\end{align}
Since $C(n,k+1)= \dfrac{n-k}{n+k+2}\matr{n+k+2}{k+1}$, the identity \eqref{eq: smaller}
can be written as
$$\matr{n+k+2}{k+1}-\dfrac{n-k}{n+k+2}\matr{n+k+2}{k+1}=\sum_{s=0}^{k} C(n,s)2^{k+1-s}.$$
Now, simplifying the left-hand side
$$\matr{n+k+2}{k+1}-\dfrac{n-k}{n+k+2}\matr{n+k+2}{k+1}=\dfrac{2k+2}{n+k+2}\matr{n+k+2}{k+1}=2\matr{n+k+1}{k},$$
we obtain the desired identity
\[ \matr{n+k+1}{k}=\sum_{s=0}^{k} C(n,s)2^{k-s}. \qedhere \]
\end{proof}

\begin{example}
\begin{align*}
& \matr{7}{3} = \sum_{s=0}^3 C(3,s)2^{3-s} = 1 \times 8 + 3 \times 4 + 5 \times 2 + 5 =35,\allowdisplaybreaks\\
& \matr{8}{3} = \sum_{s=0}^3 C(4,s)2^{3-s}= 1 \times 8 + 4 \times 4
+ 9 \times 2 + 14 =56.
\end{align*}
\end{example}

\subsection{Catalan trapezoid}

As a generalization of Catalan triangle, we define a {\em Catalan trapezoid} by considering  a trapezoidal array of numbers with $m$ complete columns $(m \ge 1)$. Let the entry in the $n^{\textrm{th}}$ row and $k^\textrm{th}$ column of the array be denoted by $C_m(n,k)$ for $0 \leq k \leq m+n-1$. Set the entries of the first row to be $C_m(0,0)=C_m(0,1)= \cdots = C_m(0,m-1)=1$, and then each subsequent entry is the sum of the entry above it and the entry to the left as in the case of Catalan triangle.
For example, when $m=3$, we obtain

\begin{equation} \label{CTR} \begin{array}{cccccccccc}
1&1&1\\
1 & 2&3&3\\
1 & 3 & 6&9&9\\
1 & 4 & 10 & 19&28&28\\
1 & 5 & 15 & 34 & 62&90&90\\
1 & 6 & 21 & 55 & 117 & 207&297&297\\
1 & 7 & 28 & 83 & 200& 407 & 704&1001&1001\\
1 & 8 & 36 & 119 & 319 & 726 & 1430 & 2431 & 3432 & 3432\\
\vdots & \vdots &\vdots &\vdots &\vdots &\vdots &\vdots & \vdots& \vdots& \ddots\\

\end{array}
\end{equation}

Alternatively, the numbers $C_m(n,k)$ can be defined in the following way.

\begin{definition}

For an integer $m \ge 1$, set $C_{1}(n,k)=C(n,k)$ for $0 \le k \le n$ and $C_{2}(n,k)=C(n+1,k)$ for $0 \le k \le n+1$, and define inductively
\begin{align} \label{eq: Catalan tra}
C_{m}(n+1,k)= \begin{cases}
\matr{n+1+k}{k} & \text{ if } 0 \le k < m, \\
\matr{n+1+k}{k} -\matr{n+m+1+k-m}{k-m} & \text{ if } m \le k \le n+m, \\ \qquad
0 & \text{ if } n+m < k.
\end{cases}
\end{align}
\end{definition}

Using the numbers $C_m(n,k)$, we prove the following theorem.
\begin{theorem} For any triple of integers $(m,k,n)$ such that $1 \le m \le k \le n+m$, we have
\begin{align} \label{eq: geometric}
\binom {n+k+1}{k}=\displaystyle\sum_{s=0}^{k}C(n,s)2^{k-s}=
\displaystyle\sum_{s=0}^{k-m}C(n+m,s)2^{k-m-s}+
\displaystyle\sum_{s=0}^{m-1}C(n+1+s,k-s).
\end{align}
\end{theorem}

\begin{proof}
By Theorem \ref{thm: larger}, the equation \eqref{eq: Catalan tra} can be re-written as follows:
\begin{align*}
C_{m}(n+1,k)= \begin{cases}
\displaystyle\sum_{s=0}^k C(n,s)2^{k-s} & \text{ if } 0 \le k < m, \allowdisplaybreaks\\
\displaystyle\sum_{s=0}^{k}C(n,s)2^{k-s}-\displaystyle\sum_{s=0}^{k-m}C(n+m,s)2^{k-m-s} & \text{ if } m \le k \le n+m, \allowdisplaybreaks\\
0 & \text{ if } n+m < k.
\end{cases}
\end{align*}

On the other hand, for  $m \le k \le n+m$, we have
$$ \matr{n+1+k}{k}-\matr{n+m+1+k-m}{k-m}=\displaystyle\sum_{s=0}^{m-1}C(n+1+s,k-s).$$
Thus we obtain
$$
\displaystyle\sum_{s=0}^{k}C(n,s)2^{k-s}-\displaystyle\sum_{s=0}^{k-m}C(n+m,s)2^{k-m-s}=
\displaystyle\sum_{s=0}^{m-1}C(n+1+s,k-s)$$
for $m \le k \le n+m$. This completes the proof.
\end{proof}

By specializing \eqref{eq: geometric} at $m=1$, we obtain different expressions of a binomial
coefficient $\matr{n+k+1}{k}=\matr{n+1+k}{k}$ as a $2$-power weighted sum of the Catalan triangle along the $n+1^{\textrm{st}}$ row. (cf. \eqref{eq-nk})

\begin{corollary} \label{cor: twisted}
We have the following identities: For $k \ge 1$,
\begin{equation}\label{eq-n1k}
\begin{aligned}
\matr{n+1+k}{k} & = \displaystyle\sum_{s=0}^{k}C(n,s)2^{k-s}  = \displaystyle\sum_{s=0}^{k-1}C(n+1,s)2^{k-1-s}+ C(n+1,k) \\
& = \sum_{s=0}^{k} C(n+1,s)2^{\max(k-1-s,0)}.
\end{aligned}
\end{equation}
\end{corollary}

Note that, combining Theorem \ref{thm: main1} and Corollary \ref{cor: twisted}, we have proven Theorem \ref{thm: smaller}.

\begin{remark}
The identities in Corollary 2.9 can be interpreted combinatorially, and one can prove them bijectively. As it reveals combinatorics behind the identities, we sketch a bijective proof of the first identity $\matr{n+1+k}{k}  = \displaystyle\sum_{s=0}^{k}C(n,s)2^{k-s}$.
\begin{proof}[Bijective proof]\hskip -0.15 cm  \footnote{This proof was communicated to us by Jang Soo Kim. We thank him for allowing us to use his proof.}
We interpret $\tiny \matr{n+1+k}{k}$ as the number of the lattice paths from $(0,0)$ to $(n+1+k, n+1-k)$ with up-steps $(1,1)$ and down-steps $(1,-1)$. For such a path $\lambda$, let $P_0, P_1, \dots , P_s$ be the intersections of $\lambda$ with the $x$-axis, decomposing $\lambda$ into $s+1$ parts $D_0, D_1, \dots , D_s$.  Then the last part $D_s$ of $\lambda$ from $P_s$ to the end point $(n+1+k, n+1-k)$ is a Dyck path, i.e. it stays at or above the $x$-axis.

We remove from each of $D_i$, $i=0, \dots , s-1$, the first step and the last step and from $D_s$ the first step, and flip them, if necessary, to obtain Dyck paths, and concatenate them, putting a step $(1,1)$, between them. The resulting path is a Dyck path from $(0,0)$ to $(n+k-s,n-k+s)$. This gives a $2^s$-to-$1$ map from the set of the lattice paths from $(0,0)$ to $(n+1+k, n+1-k)$ to the set of the Dyck paths from  $(0,0)$ to $(n+k-s,n-k+s)$. Since $C(n,k-s)$ is equal to the number of the Dyck paths from $(0,0)$ to $(n+k-s,n-k+s)$, we have proven
\[\matr{n+1+k}{k}  = \displaystyle\sum_{s=0}^{k}C(n,k-s)2^{s}= \displaystyle\sum_{s=0}^{k}C(n,s)2^{k-s}. \qedhere \]
\end{proof}
\end{remark}




\subsection{Catalan triangle polynomials} \label{subsec:Cpoly}

The identities in the previous subsections naturally give rise to the following definition.

\begin{definition} For $0 \le k \le n$, we define the {\it $(n,k)^{\textrm{th}}$ Catalan triangle polynomial} $\F_{n,k}(x)$ by
\begin{align} \label{eq: n,k Cat tr pol}
\F_{n,k}(x) & = \sum_{s=0}^{k} C(n,s)x^{k-s} =\sum_{s=0}^{k-1} C(n,s)x^{k-s}+C(n,k).
\end{align}
Note that the degree of $\F_{n,k}$ is $k$. 
\end{definition}

Evaluations of $\F_{n,k}(x)$ at the first three  nonnegative integers are as follows:
\begin{itemize}
\item $\F_{n,k}(0)= C(n,k)$
\item $\F_{n,k}(1)= C(n+1,k)=C_2(n,k)$
\item $\F_{n,k}(2)= \matr{n+k+1}{k}=C_n(n+1,k)$
\end{itemize}

Clearly,
\begin{align*}
& \matr{n}{k}=\matr{n-1}{k-1}+\matr{n-1}{k}, \ \  C(n,k) = C(n,k-1) + C(n-1,k),\\
& \F_{n,k}(d) = \F_{n,k-1}(d)+\F_{n-1,k}(d) \quad \text{ for any } d \in \mathbb{Z}.
\end{align*}

Let us recall the description of $C(n,k)$ in terms of binomial coefficients:
\begin{align} \label{eq: n,k binom}
C(n,k)= 2 \matr{n+k}{k}-(2-1)\matr{n+k+1}{k}.
\end{align}

\begin{theorem} \label{thm: bino gen}
For any $d \in \mathbb{Z}_{\ge 1}$, we have
$$C(n,k)= d  \, \F_{n-1,k}(d)-(d-1)\F_{n,k}(d)$$
which recovers \eqref{eq: n,k binom} when $d=2$.
\end{theorem}

\begin{proof} Since $\F_{n,k}(d)=\sum_{s=0}^{k-1} C(n,s)d^{k-s}+C(n,k)$, we have
$$ \dfrac{\F_{n,k}(d)-C(n,k)}{d}=\sum_{s=1}^{k-1} C(n,s-1)d^{k-s}.$$
Hence the equation \eqref{sum rule} and the fact that $C(n,0)=C(n-1,0)=1$  yield
\begin{align} \label{eq: bino gen}
\F_{n,k}(d)- \dfrac{\F_{n,k}(d)-C(n,k)}{d} = \F_{n-1,k}(d).
\end{align}
By multiplying \eqref{eq: bino gen} by $d$, our assertion follows.
\end{proof}

With the consideration of Corollary \ref{cor: twisted}, we define a natural variation of $\F_{n,k}(x)$.

\begin{definition} For $0 \le k \le n$, we define the {\it modified $(n,k)^{\textrm{th}}$ Catalan triangle polynomial} in the following way:
\begin{align} \label{eq: twisted n,k Cat tr pol}
\wF_{n,k}(x) & = \sum_{s=0}^{k} C(n+1,s)\, x^{\max(k-1-s,0)}.
\end{align}
Note that the degree of $\wF_{n,k}(x)$  is $k-1$. 
\end{definition}

Evaluations of $\wF_{n,k}(x)$ at the first three nonnegative integers are as follows:

\begin{itemize}
\item $\wF_{n,k}(0)= C(n+1,k-1)+C(n+1,k)$
\item $\wF_{n,k}(1)= C(n+2,k)$
\item $\wF_{n,k}(2)= \matr{n+k+1}{k}$
\end{itemize}

Similarly,
$$\wF_{n,k}(d) = \wF_{n,k-1}(d)+\wF_{n-1,k}(d) \quad \text{ for any } d \in \mathbb{Z}.$$

We have the same result in Theorem \ref{thm: bino gen} for $\wF_{n,k}(d)$ also:

\begin{theorem} \label{thm: bino gen2}
For any $d \in \mathbb{Z}_{\ge 1}$, we have
$$C(n,k)= d \, \wF_{n-1,k}(d)-(d-1)\wF_{n,k}(d)$$
which recovers \eqref{eq: n,k binom} when $d=2$.
\end{theorem}
%

Let $\sigma_n$ be the number of $n$-celled {\em stacked directed animals} in a square lattice. See \cite{BR} for definitions. The sequence $( \sigma_n )_{n\ge 0}$ is listed as $A059714$ in OEIS.
We conjecture   \[ \sigma_n=\wF_{n,n}(3)  \quad \text{ for } n\ge 0 .\]
We expect that one can find a  direct,  combinatorial proof of this conjecture. A list of the numbers $\sigma_n$ is below, and  the conjecture is easily verified for theses numbers in the list:
\begin{align*}
\sigma_0=1, \ \sigma_1=3,\ \sigma_2=11,\ \sigma_3=44,\ \sigma_4=184,\ \sigma_5=789,\ \sigma_6=3435, \\ \sigma_7=15100, \ \sigma_8=66806,\ \sigma_9=296870,\ \sigma_{10}=1323318 , \ \sigma_{11}=5911972.
\end{align*}
For example, when $n=7$, we have the sequence $( C(8,k), \ 0 \le k \le 7 )$ equal to
\[ ( 1, \ 8, \ 35, \ 110, \ 275, \ 572, \ 1001, \ 1430  ),\] and compute
\[ \wF_{7,7}(3)= 3^6+8 \cdot 3^5+ 35 \cdot 3^4+110 \cdot 3^3 +275 \cdot 3^2 + 572 \cdot 3 + 1001 + 1430 =15100 .\]

\medskip

\section{Alternating Jacobsthal triangle} \label{sec:AJ}

In the previous section, the binomial coefficient $\binom {n+k+1} k$ is written as sums along the $n^{\mathrm{th}}$ and the $n+1^{\mathrm{st}}$ row of the Catalan triangle, respectively. In this section, we consider other rows of the Catalan triangle as well and obtain a more general result. In particular, the $n+k^{\mathrm{th}}$ row will produce a canonical sequence of numbers, which form the alternating Jacobsthal triangle. We study some subsequences of the triangle and their generating functions in the subsections.

\medskip

Define $A(m,t) \in \mathbb Z$ recursively for $m\ge t \ge 0$ by
\begin{align} \label{A-form}
A(m,0)=1, \quad   A(m,t)=A(m-1,t-1) - A(m-1,t).
\end{align}
Here we set $A(m,t)=0$ when $t>m$. Then, by induction on $m$, one can see that  \[ \displaystyle\sum_{t=1}^{m}A(m,t)=1 \quad \text{ and } \quad A(m,m)=1 .\]

Using the numbers $A(m,t)$, we prove the following theorem which is  a generalization of the identities in Corollary \ref{cor: twisted}:

\begin{theorem} \label{thm: along n+a}
For any $n > k \ge m \ge t \ge 1$, we have
\begin{align}\label{eq-nmk}
\matr{n+k+1}{k} & = \sum_{s=0}^{k-m}C(n+m,s)2^{k-m-s} + \sum_{t=1}^{m}A(m,t)C(n+m,k-m+t).
\end{align}
\end{theorem}

\begin{proof}
We will use an induction on $m \in \mathbb{Z}_{\ge 1}$. For $m=1$, we already proved the identity in Corollary \ref{cor: twisted}.
Assume that we have the identity \eqref{eq-nmk} for some $m \in \mathbb{Z}_{\ge 1}$. By specializing \eqref{eq: geometric} at $m+1$, we have
\begin{align*}
\matr{n+k+1}{k}&  =\sum_{s=0}^{k-m-1}C(n+m+1,s)2^{k-m-1-s}+
\displaystyle\sum_{s=0}^{m}C(n+1+s,k-s) \allowdisplaybreaks\\
&  =\sum_{s=0}^{k-m-1}C(n+m+1,s)2^{k-m-1-s}+
\displaystyle\sum_{s=0}^{m-1}C(n+1+s,k-s) \allowdisplaybreaks\\  & \quad + C(n+m+1,k-m).
\end{align*}
By the induction hypothesis applied to \eqref{eq: geometric}, we have
$$\sum_{s=0}^{m-1}C(n+1+s,k-s)=\sum_{t=1}^{m}A(m,t)C(n+m,k-m+t).$$
Then our assertion follows from the fact that \[ C(n+m,k-m+t)=C(n+m+1,k-m+t)-C(n+m+1,k-m+t-1). \qedhere \]
\end{proof}


We obtain the triangle consisting of $A(m,t)$ ($m \ge t \ge 0$):
\begin{equation} \label{CTTT} \begin{array}{cccccccccc}
\bu{1}\\
\bu{1} & 1\\
\bu{1} & 0 & 1\\
\bu{1} & 1 & -1 & 1\\
\bu{1} & 0 & 2 & -2 & 1\\
\bu{1} & 1 & -2 & 4 & -3 & 1\\
\bu{1} & 0 & 3 & -6 & 7 & -4 & 1\\
\bu{1} & 1 & -3 & 9 & -13 & 11 & -5 & 1\\
\bu{1} & 0 & 4 & -12 & 22 & -24 & 16 & -6 & 1\\
\vdots & \vdots &\vdots &\vdots &\vdots &\vdots &\vdots & \vdots& \vdots& \ddots\\

\end{array}
\end{equation}

The triangle in \eqref{CTTT} will be called the {\em alternating Jacobsthal triangle}.
The $0^{\mathrm{th}}$ column is colored in blue to indicate the fact that some formulas do not take entries from this column.

\begin{example}
For $m=3$, we have
\begin{align*}
\matr{n+k+1}{k} & = \sum_{s=0}^{k-3}C(n+3,s)2^{k-3-s} + C(n+3,k-2)-C(n+3,k-1)+C(n+3,k).
\end{align*}
\end{example}

By specializing \eqref{eq-nmk} at $m=k$, the $k^{\textrm{th}}$ row of alternating Jacobsthal triangle and $n+k^{\textrm{th}}$ row of Catalan triangle yield the binomial coefficient $\matr{n+k+1}{k}$:

\begin{corollary} For any $n > k$, we have
\begin{equation}\label{eq:combAC}
\matr{n+k+1}{k}  = \sum_{t=0}^{k}A(k,t)\, C(n+k,t).
\end{equation}
\end{corollary}

\begin{example} \hfill
\begin{align*}
\matr{8}{3} & = A(3,0)C(7,0)+ A(3,1)C(7,1)+ A(3,2)C(7,2)+ A(3,3)C(7,3)  \allowdisplaybreaks\\          
& = 1 \times 1+ 1 \times 7 - 1 \times 27  + 1 \times  75 =56. \allowdisplaybreaks\\
\matr{9}{3} & = A(3,0)C(8,0)+ A(3,1)C(8,1)+ A(3,2)C(8,2)+ A(3,3)C(8,3)  \allowdisplaybreaks\\          
& = 1 \times 1+ 1 \times 8 - 1 \times 35  + 1 \times  110 =84. \allowdisplaybreaks\\
\matr{9}{4} & = A(4,0)C(8,0)+ A(4,1)C(8,1)+ A(4,2)C(8,2)+ A(4,3)C(8,3)+ A(4,3)C(8,3)  \allowdisplaybreaks\\          
& = 1 \times 1+ 0 \times 8 + 2 \times 35  - 2 \times  110 + 1 \times 275=126.
\end{align*}
\end{example}




\subsection{Generating function} The numbers $A(m,t)$ can be encoded into a generating function in a standard way. Indeed, from \eqref{A-form}, we obtain
\begin{align} \label{in}
A(m,t)&= \sum_{k=t-1}^{m-1} (-1)^{m-1-k} A(k,t-1) \\ &=A(m-1,t-1)-A(m-2,t-1)-\cdots +(-1)^{m-t}A(t-1,t-1). \nonumber
\end{align}

\begin{lemma}
We have
\begin{align} \label{eq:generating function}
\dfrac{1}{(1-x)(1+x)^t}=\sum_{m=t}^{\infty}A(m,t)x^{m-t}.
\end{align}
\end{lemma}

\begin{proof}
When $t=0$, we have $\dfrac{1}{1-x}=1+x+x^2+ \cdots = \sum_{m=0}^{\infty}A(m,0)x^{m}$. Inductively, when $t>0$, we have
\[ \frac{1}{(1-x)(1+x)^t}= \frac 1 {(1+x)} \  \frac{1}{(1-x)(1+x)^{t-1}} = (1-x+x^2- \cdots ) \sum_{m=t-1}^{\infty}A(m,t-1)x^{m-t+1}. \] Then we obtain \eqref{eq:generating function} from \eqref{in}.
\end{proof}

\subsection{Subsequences} The alternating Jacobsthal triangle has various subsequences with interesting combinatorial interpretations. 

First, we write

$$\dfrac{1}{(1-x)(1+x)^t}=\sum_{m \ge 0}a_{m+1,t} x^{2m}-\sum_{m \ge 0}b_{m+1,t} x^{2m+1}$$
to define the subsequences $\{ a_{m,t} \}$ and $\{ b_{m,t} \}$ of $\{ A(m,t)\}$. Then we have
\[ a_{m,t}= A(t+2m-2,t) \quad \text{ and } \quad b_{m,t}=-A(t+2m-1,t) .\]
Clearly, $a_{m,t},b_{m,t} \ge 0$. Using \eqref{in}, we obtain
$$a_{m,t}= \sum_{k=1}^{m} a_{k,t-1}+\sum_{k=1}^{m-1} b_{k,t-1} \quad \text{ and } \quad b_{m,t}= \sum_{k=1}^{m} b_{k,t-1}+\sum_{k=1}^{m} a_{k,t-1}.$$



 It is easy to see that $a_{n,2}=n$ and $b_{n,2}=n$.
Then we have $$a_{n,3}= \sum_{k=1}^{n} a_{k,2}+\sum_{k=1}^{n-1} b_{k,2}= \dfrac{n(n+1)}{2}+ \dfrac{n(n-1)}{2}=n^2.$$
Similarly,
$$b_{n,3}= \sum_{k=1}^{n} b_{k,2}+\sum_{k=1}^{n} a_{k,2}= \dfrac{n(n+1)}{2} + \dfrac{n(n+1)}{2}=n(n+1).$$
We compute more and obtain
\begin{align*}
a_{n,4} & =  \dfrac{n(n+1)(4n-1)}{6},  &
b_{n,4} & = \dfrac{n(n+1)(4n+5)}{6}, \allowdisplaybreaks\\
a_{n,5} & = \dfrac{n(n+1)(2n^2+2n-1)}{6}, &
b_{n,5} & =  \dfrac{n(n+1)^2(n+2)}{3}.
\end{align*} Note also that
$$\dfrac{1}{(1+x)(1-x)^t}=\sum_{m \ge 0}a_{m+1,t} x^{2m}+\sum_{m \ge 0}b_{m+1,t} x^{2m+1}.$$

Next, we define $$B(m,t)=A(m,m-t) \qquad \text{ for } m \ge t$$
to obtain the triangle
\begin{equation}  \label{Bt} \begin{array}{cccccccccc}
\bu{1}\\
1 & \bu{1}\\
1 & 0 &  \bu{1}\\
1 & -1 & 1 & \bu{1}\\
1 & -2 & 2 & 0 & \bu{1}\\
1 & -3 & 4 & -2 & 1 & \bu{1}\\
1 & -4 & 7 & -6 & 3 & 0 & \bu{1}\\
1 & -5 & 11 & -13 & 9 & -3 & 1 & \bu{1}\\
1 & -6 & 16 & -24 & 22 & -12 & 4 & 0 & \bu{1}\\
\vdots & \vdots &\vdots &\vdots &\vdots &\vdots &\vdots & \vdots& \vdots&  \ddots\\

\end{array}
\end{equation}

\begin{lemma} \label{lem-bk}
For $m \ge t$, we have
\[ B(m,t)=1 - \sum_{k=t}^{m-1} B(k,t-1) . \]
\end{lemma}

\begin{proof}
We use induction on $m$. When $m=t$, we have $B(m,t)=A(m,0)=1$. Assume that the identity is true for some $m \ge t$. Since we have
\begin{align*} B(m,t) &= A(m,m-t)=A(m-1,m-t-1)-A(m-1,m-t) \\ &=B(m-1,t)-B(m-1,t-1), \end{align*} we obtain
\begin{align*}
\sum_{k=t}^{m-1} B(k,t-1)+B(m,t) &= \sum_{k=t}^{m-1} B(k,t-1)+  B(m-1,t)-B(m-1,t-1) \\
&= \sum_{k=t}^{m-2} B(k,t-1)+ B(m-1,t) =1
\end{align*}
by the induction hypothesis.
\end{proof}

Using Lemma \ref{lem-bk},  one can derive the following formulas:
\begin{itemize}
\item $B(n,0)=A(n,n)=1$ and $B(n,1)=A(n,n-1)=2-n$,
\item $B(n,2)=A(n,n-2)=4+\dfrac{n(n-5)}{2}$,
\item $B(n,3)=A(n,n-3)=8-\dfrac{n(n^2-9n+32)}{6}$.
\end{itemize}

We consider the columns of the triangle \eqref{Bt} and let $c_{m,t}=(-1)^t B(m+t+1,t)$ for each $t$ for convenience. Then the sequences $( c_{m,t} )_{m\ge 1}$ for first several $t$'s appear in the OEIS. Specifically, we have:

\begin{itemize}
\item $( c_{m,2} )= ( 2,4,7,11,16,22, \dots )$ corresponds to $A000124$,
\item $( c_{m,3} )= ( 2,6,13,24,40,62, \dots )$ corresponds to $A003600$,
\item $( c_{m,4} )= ( 3,9,22,46,86, 148, \dots )$ corresponds to $A223718$,
\item $( c_{m,5} )= ( 3,12,34,80,166,314, \dots )$ corresponds to $A257890$,
\item $( c_{m,6} )= ( 4,16,50,130,296,610, \dots )$ corresponds to $A223659$.
\end{itemize}

\subsection{Diagonal sums} As we will see in this subsection, the sums along lines of slope $1$ in the alternating Jacobsthal triangle are closely related to Fibonacci numbers. We begin with fixing a notation.
For $s \ge 0$, define \[ B_s= \sum_{t+m-2=s,\, t>0}A(m,t) . \]
Using the generating function  \eqref{eq:generating function}, we have
$$F(x):=\sum_{t=1}^{\infty}\dfrac{x^{2t-2}}{(1-x)(1+x)^t} = \sum_{s=0}^{\infty} B_s x^{s}.$$
Then we obtain
\begin{align*}
(1-x) x^2F(x) & =\sum_{t=1}^{\infty} \left( \dfrac{x^{2}}{1+x} \right)^t = \dfrac{x^2}{1+x-x^2}
\end{align*}
and the formula
\begin{align} \label{eq: simple formula}
F(x) = \dfrac{1}{(1-x)(1+x-x^2)}.
\end{align}
It is known that the function $F(x)$ is the generating function of the sequence of the alternating sums of the Fibonacci numbers; precisely, we get
\begin{equation} \label{eq:ff} B_s= \sum_{k=1}^{s+1} (-1)^{k-1} \, \mathrm{Fib}(k) =1+(-1)^{s} \, \mathrm{Fib}(s) \quad (s \ge 0), \end{equation}
where $( \mathrm{Fib}(s) )_{s\ge 0} $ is the Fibonacci sequence. (See A119282 in OEIS.) From the construction, the following is obvious:
$$B_{s+1}=-B_s+B_{s-1}+1 \ (s \ge 1) \quad \text{and} \quad B_0=1, \ B_1=0.$$

\section{$q$-deformation}

In this section, we study a $q$-deformation of the Fibonacci and Jacobsthal numbers by putting the parameter $q$ into
the identities and generating functions we obtained in the previous section. We also obtain a family of generating functions
of certain sequences by expanding the $q$-deformation of the generating function of the numbers $A(m,t)$ in terms of $q$.

\subsection{$q$-Fibonacci numbers}
For $s \ge 0$, define
 \[ B_s(q)=  \sum_{t+m-2=s, \, t>0}A(m,t)q^{m-t} \in \Z[q] . \] Then
we obtain \begin{align*}
F(x,q) & :=\sum_{t=1}^{\infty}\dfrac{x^{2t-2}}{(1-qx)(1+qx)^t} \\ &= \sum_{s=0}^{\infty} B_s(q)x^{s} = \dfrac{1}{(1-qx)(1+qx-x^2)} .
\end{align*}
Note that $$B_{s+1}(q)=-qB_s(q)+B_{s-1}(q)+q^s \ (s \ge 1) \quad \text{and} \quad B_0(q)=1, \, B_1(q)=0.$$

Motivated by \eqref{eq:ff}, we define a $q$-analogue of Fibonacci number by
\[ \wB_s(q):=(-1)^s B_s(q)+ (-1)^{s+1}q^s =  \sum_{t+m-2=s, \, t>0} |A(m,t)|q^{m-t} +(-1)^{s+1}q^s .\]
In particular, we have
\begin{align*}
\wB_1(q)&=q, \quad \wB_2(q)=1, \quad \wB_3(q)=q^3+q, \quad \wB_4(q)=2q^2+1, \quad \wB_5(q)=q^5+2q^3+2q, \\ \wB_6(q)&=3q^4+4q^2+1, \quad \wB_7(q)=q^7+3q^5+6q^3+3q, \quad \wB_8(q)=4q^6+9q^4+7q^2+1.
\end{align*}
These polynomials can be readily read off from the alternating Jacobsthal triangle \eqref{CTTT}.
We observe that $\wB_{2s}(q) \in \Z_{\ge 0}[q^2]$ and $\wB_{2s+1}(q) \in \Z_{\ge 0}[q^2]q$ and that $\wB_s(q)$ is weakly unimodal.

Note that we have
\begin{align*}
  & F(x,q)-\dfrac{1}{1-qx}  = \dfrac{1}{(1-qx)(1+qx-x^2)}-\dfrac{1}{1-qx} = \dfrac{x^2-qx}{(1-qx)(1+qx-x^2)} \\ &= \sum_{s=0}^\infty B_s(q) x^s - \sum_{s=0}^\infty (qx)^s = \sum_{s=0}^{\infty} \wB_s(q)(-x)^s.
\end{align*}
Thus the generating function $CF(x,q)$ of $\wB_s(q)$ is given by
\begin{align}
CF(x,q) & \seteq \sum_{s=0}^{\infty} \wB_s(q)x^s =\dfrac{x^2+qx}{(1+qx)(1-qx-x^2)}.
\end{align}

\begin{remark}
The well-known Fibonacci polynomial $\mathcal F_s(q)$ can be considered as a different $q$-Fibonacci number whose generating function is given by
$$\sum_{t=1}^{\infty} \dfrac{x^{2t-2}}{(1-qx)^t}= \dfrac{1}{1-qx-x^2}= \sum_{s=0}^{\infty}\mathcal{F}_s(q)x^s.$$ Recall that the polynomial $\mathcal F_s(q)$ can be read off from the Pascal triangle.
When $q=2$, the number $\mathcal F_s(2)$ is nothing but the $s^{\mathrm{th}}$ Pell  number. On the other hand, it does not appear that the sequence \[ ( \wB_s(2) )_{s \ge 1} =( 2,1,10,9,52,65,278,429,1520, \dots ) \] has been studied in the literature.


\end{remark}





\subsection{$q$-Jacobsthal numbers} \label{sec:J}

Recall that the Jacobsthal numbers $J_m$ are defined recursively by $J_m=J_{m-1}+2J_{m-2}$ with $J_1=1$ and $J_2=1$. Then the Jacobsthal sequence $(J_m )$ is given by
$$( 1, 1, 3, 5, 11, 21, 43, 85, 171, \dots ). $$

Consider the function
\begin{align} \label{eq: AL JA}
Q(x,q):=\sum_{t=1}^{\infty}\dfrac{x^{t}}{(1-qx)(1+qx)^t}.
\end{align}
Define \[ H_m(q):= \sum_{t=1}^m A(m,t) q^{m-t}  \qquad \text{ for } m \ge 1 . \] For example, we can read off
$$ H_5(q)=q^{4}-2q^3+4q^2-3q+1$$
from the alternating Jacobsthal triangle \eqref{CTTT}.
Using \eqref{eq:generating function}, we obtain
\begin{align*}
Q(x,q) &= \sum_{t=1}^\infty \sum_{m=t}^\infty A(m,t) q^{m-t}x^m \\
&= \sum_{m=1}^\infty \sum_{t=1}^m A(m,t) q^{m-t} x^m = \sum_{m=1}^\infty H_m(q) x^m.
\end{align*}
A standard computation also yields
\begin{equation} \label{qq} Q(x,q)= \frac {x}{(1-qx)(1+(q-1)x)} .\end{equation}

By taking $q=0$ or $q=1$, the equation \eqref{qq} becomes $\dfrac{x}{1-x}$. On the other hand, by taking $q=-1$, the equation
\eqref{qq} becomes $$\dfrac{x}{(1+x)(1-2x)},$$
which is the generating function of the Jacobsthal numbers $J_m$. That is, we have
\begin{itemize}
\item $H_m(0)=H_m(1)=1$ for all $m$,
\item $H_m(-1)$ is the $m^{\mathrm{th}}$ Jacobsthal number $J_m$ for each $m$.
\end{itemize}
Since $J_m=H_m(-1)=\sum_{t=1}^m |A(m,t)|$, we see that an alternating sum of the entries along a row of the triangle \eqref{CTTT} is equal to a Jacobsthal number.

Moreover, we have a natural $q$-deformation $J_m(q)$ of the Jacobsthal number $J_m$, which is defined by \[ J_m(q) := H_m(-q) = \sum_{t=1}^m |A(m,t)| q^{m-t} .\]
For example, we have
\[ J_3(q)=q^2+q+1, \quad J_4(q)=2q^2+2q+1, \quad J_5(q)=q^4+2q^3+4q^2+3q+1 . \]
Note that $J_m(q)$ is weakly unimodal.
We also obtain
$$Q(x,-q)= \dfrac{x}{(1+qx)(1-(q+1)x)} = \sum_{m=1}^{\infty} J_m(q)x^{m}.$$



The following identity is well-known (\cite{HBJ,Hor}):
$$J_m= \sum_{r=0}^{\lfloor (m-1)/2 \rfloor} \matr{m-r-1}{r}2^r.$$
Hence we have
\begin{align*}
J_m= \sum_{r=0}^{\lfloor (m-1)/2 \rfloor} \matr{m-r-1}{r}2^r & = \sum_{t=1}^m |A(m,t)|=H_m(-1)=J_m(1).
\end{align*}

\begin{remark}
In the literature, one can find different Jacobsthal polynomials. See \cite{Koshy}, for example.
\end{remark}

\subsection{A family of generating functions}
Now let us expand $Q(x,q)$ with respect to $q$ to define the functions $L_\ell(x)$:
$$ Q(x,q)=\sum_{t=1}^{\infty}\dfrac{x^{t}}{(1-qx)(1+qx)^t}= \sum_{\ell=0}^{\infty} L_\ell(x)q^{\ell}.   $$

\begin{lemma} \label{lem-l}
For $\ell \ge 0$, we have $$L_{\ell+1}(x) = \dfrac{-x}{1-x} L_{\ell}(x) + \dfrac{x^{\ell+2}}{1-x}.$$
\end{lemma}

\begin{proof}
Clearly, we have $L_0(x)=\displaystyle\sum^{\infty}_{n=1} x^n =\dfrac{x}{1-x}$.
We see that
\[ \frac x{1-x}+\sum_{\ell=0}^\infty L_{\ell +1}(x) q^{\ell +1} = Q(x,q) = \frac x {(1-qx)(1+(q-1)x)} .\] On the other hand, we obtain
\begin{align*}
& \frac x{1-x} + \sum_{\ell =0}^\infty \left \{ \frac {-x}{1-x} L_\ell (x) q^{\ell +1} + \frac x {1-x} (qx)^{\ell +1} \right \} \allowdisplaybreaks\\
& =\frac x {1-x} - \frac{qx}{1-x} \cdot \frac x {(1-qx)(1+(q-1)x)} + \frac x {1-x} \cdot \frac {qx}{1-qx} \allowdisplaybreaks\\
&= \frac x {(1-qx)(1+(q-1)x)} = Q(x,q) .
\end{align*}
This completes the proof.
\end{proof}

Using Lemma \ref{lem-l}, we can compute first several $L_\ell (x)$:
\begin{itemize}
\item $L_0(x)=\displaystyle\sum^{\infty}_{n=1} x^n =\dfrac{x}{1-x}$,
\item $L_1(x)=\dfrac{-x^2}{(1-x)^2}+\dfrac{x^2}{1-x} = \dfrac{-x^3}{(1-x)^2} = -\displaystyle\sum^{\infty}_{n=1} nx^{n+2}$,
\item $L_2(x)= \dfrac{x^4}{(1-x)^3} +  \dfrac{x^3}{1-x}=\dfrac{x^3(1-x+x^2)}{(1-x)^3}.$
\item $L_3(x)= -\dfrac{x^4(1-x+x^2)}{(1-x)^4}+ \dfrac{x^4}{1-x}= - \dfrac {x^5(2-2x+x^2)}{(1-x)^4}.$
\end{itemize}

One can check that $L_2(x)$ is the generating function of the sequence $A000124$ in OEIS and that  $L_3(x)$ is the generating function of the sequence $A003600$. Note that the lowest degree of $L_{\ell}(x)$ in the power series expansion is larger than or equal to $\ell +1$.
More precisely, the lowest degree of $L_{\ell}(x)$ is $\ell+1+\delta(\ell \equiv 1 (\mathrm{mod} \ 2))$.

\section{$k$-analogue of $q$-deformation}

In this section, we consider $k$-analogues of the $q$-deformations we introduced in the previous section. This construction, in particular, leads to a $k$-analogue of the alternating Jacobsthal triangle  for each $k \in \mathbb Z \setminus \{ 0 \}$. Specializations of this construction at some values of $k$ and $q$ produce interesting combinatorial sequences.

\medskip

Define $A_{k}(m,t)$ by \[  A_k(m,0)=k^{\lfloor m/2\rfloor} \quad \text{ and } \quad    A_k(m,t)=A_k(m-1,t-1) - A_k(m-1,t).\]
Then we have
\[ \frac {1}{(1-kx^2)(1+x)^{t-1}} = \sum_{m=t}^\infty A_k(m,t)x^{m-t} \]
in the same way as we obtained \eqref{eq:generating function}.
As in Section \ref{sec:J}, we also define \[ H_{k,m}(q)=\sum_{t=1}^m A_{k}(m,t) q^{m-k} .\]
We obtain the generating function
$Q_k(x,q)$ of $H_{k,m}(q)$ by
\[ Q_k(x,q):=\sum_{t=1}^{\infty}\dfrac{x^{t}}{(1-k q^2x^2)(1+qx)^{t-1}} = \dfrac{x(1+qx)}{(1-kq^2x^2)(1+(q-1)x)} = \sum_{m=1}^{\infty} H_{k,m}(q)x^{m}.
\]
In particular, when $q=1$, we have
\[Q_k(x,1)=\sum_{t=1}^{\infty}\dfrac{x^{t}}{(1-kx^2)(1+x)^{t-1}} = \dfrac{x(1+x)}{1-kx^2}= \sum_{m=1}^{\infty} k^{m-1} \big( x^{2m-1}+x^{2m}\big). \]
Note that $H_{k,m}(1)= k^{\lfloor (m-1)/2 \rfloor}$.

Moreover, the triangle given by the numbers $A_k(m,t)$ can be considered as a $k$-analogue of the alternating Jacobsthal triangle \eqref{CTTT}. See the triangles \eqref{T1} and \eqref{T3}.  
Thus we obtain infinitely many  triangles as $k$ varies in $\mathbb Z \setminus \{0\}$. Similarly, we define a $k$-analogue of the $q$-Jacobsthal number by
\[ J_{k,m}(q) := H_{k,m}(-q), \] and the number $J_{k,m}(1)=\sum_{t=1}^m |A_{k}(m,t)|$ can be considered as the $k$-analogue of the $m^{\mathrm{th}}$ Jacobsthal number.

For example, if we take $k=2$, the polynomial $H_{2,m}(q)$ can be read off from the following triangle consisting of $A_2(m,t)$:
\begin{equation} \label{T1} \begin{array}{cccccccccccc}
\bu{1}\\
\bu{1}& 1\\
\bu{2}& 0 & 1\\
\bu{2}&2 & -1 & 1\\
\bu{4}&0 & 3 & -2 & 1\\
\bu{4}&4 & -3 & 5 & -3 & 1\\
\bu{8}&0 & 7 & -8 & 8 & -4 & 1\\
\bu{8}&8 & -7 & 15 & -16 & 12 & -5 & 1\\
\bu{16}&0 & 15 & -22 & 31 & -28 & 17 & -6 & 1\\
\bu{16}&16 & -15 & 37  &-53  & 59  &-45  &23  &-7  &1 \\
\bu{32}&0 & 31 & -52 & 90 & -112  & 104  & -68 & 30  & -8 & 1 \\
\vdots & \vdots & \vdots &\vdots &\vdots &\vdots &\vdots &\vdots & \vdots& \vdots & \vdots & \ddots\\
\end{array}
\end{equation}
We have $J_{2,m}(1)=\sum_{t=1}^m |A_{2}(m,t)|$, and the sequence \[(J_{2,m}(1) )_{m\ge 1} = (1, 1, 4, 6, 16, 28, 64, 120, \dots )\]  appears as $A007179$ in OEIS. As mentioned in the introduction, this sequence has the interpretation as the numbers of equal
dual pairs of some integrals studied in \cite{Heading}. (See Table 1 on p.365 in \cite{Heading}.)

Define $B_k(m,t) = A_k(m,m-t)$. Then we obtain the following sequences from \eqref{T1} which appear in OEIS:
\begin{itemize}
\item $(B_2(m,2) )_{m \ge 3} = (2,3,5,8,12,17,23,30,\dots ) \leftrightarrow A002856, A152948$,
\item $(-B_2(m,3) )_{m \ge 5} = (3,8,16,28,45,68, \dots ) \leftrightarrow A254875$.
\end{itemize}
We also consider diagonal sums and find
\begin{itemize}
\item the positive diagonals  $$\left(\displaystyle\sum_{m+t=2s, \, t>0}A_2(m,t) \right)_{s \ge 1}$$ corresponds
$$  (1, 3, 8, 21, 55, 144, 377, \dots ) \leftrightarrow (\mathrm{Fib}(2s)),$$
where $\mathrm{Fib}(s)$ is the Fibonacci number;
\item the negative diagonals  $$\left(-\displaystyle\sum_{m+k=2s+1, \, t>0}A_2(m,k)\right)_{s \ge 1}$$ corresponds
$$  (0,1,5,18,57,169, \dots ) \leftrightarrow A258109,$$
whose $s^{\mathrm{th}}$ entry is the number of Dyck paths of length $2(s+1)$ and height $3$.
\end{itemize}


Similarly, when $k=-1$, we obtain the following triangle consisting of $A_{-1}(m,t)$:
\begin{equation} \label{T3} \begin{array}{cccccccccccc}
\bu{1}\\
\bu{1}& \rnum{1}\\
\bu{-1}&\rnum{0}  & 1\\
\bu{-1}&\rnum{-1} & -1& 1\\
\bu{1}&\rnum{0}  & 0 & -2 & 1\\
\bu{1}&\rnum{1}  & 0 & 2  & -3& 1\\
\bu{-1}&\rnum{0}  & 1 & -2 & 5 & -4 & 1\\
\bu{-1}&\rnum{-1} & -1& 3 & -7 & 9  & -5 & 1\\
\bu{1}&\rnum{0}  & 0 & -4 & 10& -16& 14 & -6 & 1\\
\bu{1}&\rnum{1}  & 0 & 4  & -14& 26& -30 & 20 & -7  &1 \\
\bu{-1}&\rnum{0}  & 1 & -4 & 18 & -40& 56  & -50 & 27& -8 & 1 \\
\vdots & \vdots & \vdots &\vdots &\vdots &\vdots &\vdots &\vdots & \vdots& \vdots & \vdots & \ddots\\
\end{array}
\end{equation}
We find some meaningful subsequences of this triangle and list them below.
\begin{itemize}
\item $(|A_{-1}(m,4)|)_{m \ge 4}=(1,3,5,7,10,14,18,22, \dots) =(\lfloor \binom n 2 /2 \rfloor)_{n\ge 3} \leftrightarrow A011848$,

\item $(B_{-1}(m,2))_{m \ge 5}=(2,5,9,14,20,27,\dots) \leftrightarrow A212342$,
\item $(-B_{-1}(m,3))_{m \ge 6}=(2,7,16,30,50,77, \dots) \leftrightarrow A005581$,
\item $\left ( \sum_{m+t=2s}A_{-1}(m,t) \right)_{s\ge 2}=(1,1,4,9,25,64,169,441, \dots) =( \mathrm{Fib}(n)^2)_{n \ge 1}$,
\item $\left ( \sum_{t=2}|A_{-1}(m,t)| \right)_{m\ge 2}=(1,2,3,6,13,26,51,102 ,\dots )\leftrightarrow A007910$.
\end{itemize}

Define $B_{k,s}(q)$ and $F_k(x,q)$ by $B_{k,s}(q):= \displaystyle \sum_{t+m-2=s, \, t>0}A_k(m,t)q^{m-t} \in \Z[q]$ and
\begin{align*}
F_k(x,q) & :=\sum_{t=1}^{\infty}\dfrac{x^{2t-2}}{(1-kq^2x^2)(1+qx)^{t-1}}= \dfrac{1+qx}{(1-kq^2x^2)(1+qx-x^2)} \\ & = \sum_{s=1}^{\infty} B_{k,s}(q)x^{s}.
\end{align*}
Let us consider the following to define $\wB_{k,s}(q)$:
\begin{align*}
\dfrac{1+qx}{(1-kq^2x^2)(1+qx-x^2)}- \dfrac{1+qx}{1-kq^2x^2}  = \dfrac{(1+qx)(-qx+x^2)}{(1-kq^2x^2)(1+qx-x^2)}  =\sum_{s=0}^{\infty} \wB_{k,s}(q)(-x)^{s}.
\end{align*}
Define a $k$-analogue $CF_k(x,q)$ of the function $CF(x,q)$ by
\begin{align*}
CF_k(x,q) & \seteq \dfrac{(1-qx)(qx+x^2)}{(1-kq^2x^2)(1-qx-x^2)}= \sum_{s=0}^{\infty} \wB_{k,s}(q)x^{s} .
\end{align*}
The polynomial $\wB_{k,s}(q)$ can be considered as a $k$-analogue of the $q$-Fibonacci number $\wB_s(q)$.



Finally, we define $L_{k,\ell +1}(x)$ by
$$ Q_k(x,q)= \sum_{t=1}^{\infty}\dfrac{x^{t}}{(1-kq^2x^2)(1+qx)^{t-1}}= \sum_{\ell=0}^{\infty} L_{k,\ell}(x)q^{\ell}.   $$
Then, using a similar argument as in the proof of Lemma \ref{lem-l}, one can show that
$$L_{k,\ell+1}(x) = \dfrac{-x}{1-x} L_{k,\ell}(x) + \dfrac{ k^{\lfloor (\ell+1)/2 \rfloor}  x^{\ell+2}}{1-x}.$$

\begin{remark}
We can consider the Jacobsthal--Lucas numbers and the Jacobsthal--Lucas polynomials starting with the generating function
\[ \dfrac{1+4x}{(1-x^2)(1-x)^{t-1}}, \]
and study their ($k$-analogue of) $q$-deformation.
\end{remark}

\end{document}